\documentclass[12pt]{amsart}
\usepackage{bbm,bm}
\usepackage{amssymb}
\usepackage{colordvi}
\usepackage{graphicx,xspace}
\usepackage{color}
\usepackage{xspace}

\usepackage[bookmarks=false,colorlinks,linkcolor=blue,citecolor=green]{hyperref}
\newcounter{FNC}[page]
\def\fauxfootnote#1{{\addtocounter{FNC}{2}$^\fnsymbol{FNC}$%
     \let\thefootnote\relax\footnotetext{$^\fnsymbol{FNC}$\Magenta{#1}}}}
\numberwithin{equation}{section}
\newtheorem{theorem}{Theorem}[section]

\newtheorem{defi}[theorem]{Definition}
\newtheorem{exam}[theorem]{Example}
\newtheorem{rema}[theorem]{Remark}
\author{Stefan Forcey} \address[S. Forcey]{
    Department of Mathematics\\
    The University of Akron\\
    Akron, OH 44325-4002
    }
    \email{sf34@uakron.edu}  \urladdr{http://www.math.uakron.edu/\~{}sf34/}

\author{Logan Keefe} \address[L. Keefe]{
    Department of Mathematics\\
    The University of Akron\\
    Akron, OH 44325-4002
    }

\author{William Sands} \address[W. Sands]{
    Department of Mathematics\\
    The University of Akron\\
    Akron, OH 44325-4002
    }

\title[BME Facets]{Facets of the Balanced Minimal Evolution Polytope.}

\begin{document}

\begin{abstract}

 The balanced minimal evolution (BME)
method of creating phylogenetic trees can be formulated as a linear programming
problem, minimizing an inner product over the vertices of the BME polytope. In
this paper we undertake the project of describing the facets of this polytope.
We classify and identify the combinatorial structure and geometry (facet
inequalities) of all the facets in dimensions up to 5, and classify even more
facets in all dimensions. A full set of facet inequalities would allow a full
implementation of the simplex method for finding the BME tree--although there
are reasons to think this an unreachable goal. However, our results provide the
crucial first steps for a more likely-to-be-successful program: finding
efficient relaxations of the BME polytope.
\end{abstract}

\maketitle

\section{Introduction}

\begin{figure}[b]\centering
                  \includegraphics[width=\textwidth]{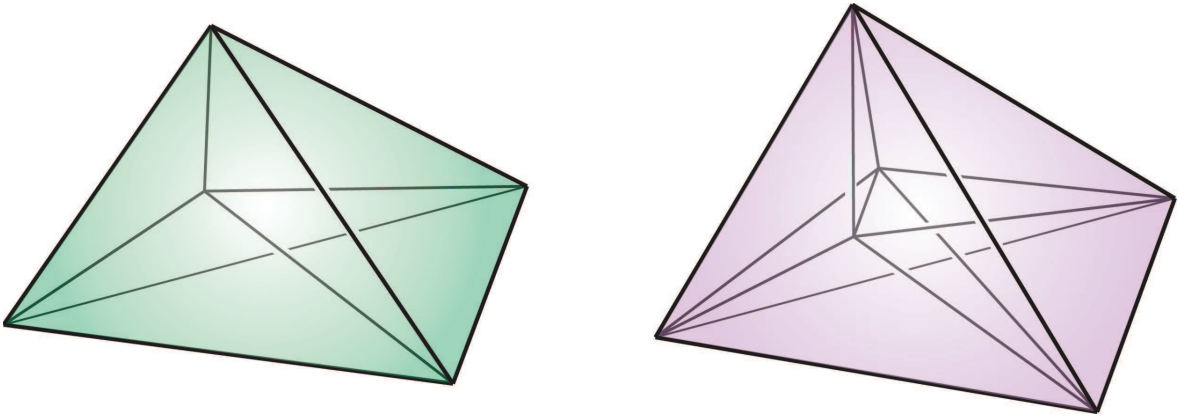}
\caption{There are two combinatorial types of facets of the
5-dimensional BME polytope, shown here as Schlegel diagrams. The
polytope on the left is the 4-simplex, and on the right is the 4d
Birkhoff polytope. The
Schlegel diagram for the latter is expanded in
Figure~\ref{f:expand}.}\label{f:facets}
\end{figure}

The goal of phylogenetics is to take a set of related items--
biological examples are usually referred to as taxa: populations,
species, individuals or genes--and to construct a branching diagram
that explains how they are related chronologically. The diagram we
will be concerned with is a binary tree with labeled leaves. In
other words, a cycle-free graph with nodes (vertices) which are
either of degree one (touching a single edge) or degree three, and
with a set of distinct items assigned to the degree one nodes--the
leaves. We study a method called \emph{balanced minimal evolution}.
This method begins with a given set of $n$ items and a symmetric (or
upper triangular) square $n\times n$ \emph{dissimilarity matrix}
whose entries are numerical dissimilarities, or distances, between
pairs of items. From the dissimilarity matrix the balanced minimal
evolution (BME) method constructs a binary tree with the $n$ items
labeling the $n$ leaves. The BME tree has the property that the
distances between its leaves most closely match the given distances
between corresponding pairs of taxa.

By ``most closely match'' in the previous paragraph we mean the
following: the reciprocals of the distances between leaves are the
components of a vector $\mathbf{c}$, and this vector minimizes the
dot product $\mathbf{c}\cdot\mathbf{d}$ where $\mathbf{d}$ is the
list of distances in the upper triangle of the distance matrix.

More precisely: Let the set of $n$ distinct species, or taxa, be
called $S.$ For convenience we will often let $S = [n] =
\{1,2,\dots,n\}.$ Let vector $\mathbf{d}$ be given, having ${n
\choose 2}$ real valued components $d_{ij}$, one for each pair
$\{i,j\}\subset S.$ There is a vector $\mathbf{c}(t)$ for each
binary tree $t$ on leaves $S,$ also having ${n \choose 2}$
components $c_{ij}(t)$, one for each pair $\{i,j\}\subset S.$ These
components are ordered in the same way for both vectors, and we will
use the lexicographic ordering: $\mathbf{d} =
\left<d_{12},d_{13},\dots,d_{1n},d_{23},d_{24},\dots,d_{n-1,n}\right>
$.

We define, following Pauplin \cite{Pauplin}:
 $$c_{ij}(t) = \frac{1}{2^l}$$
where $l$ is the number of internal nodes (degree 3 vertices) in the
path from leaf $i$ to leaf $j.$

The BME tree for the vector $\mathbf{d}$ is the binary tree $t$ that
minimizes $\mathbf{d}\cdot\mathbf{c}(t)$ for all binary trees on
leaves $S.$ The value of setting up the question in this way is that
it becomes a linear  programming problem. The convex hull of all the
vectors $\mathbf{c}(t)$ for all binary trees $t$ on $S$ is a
polytope BME$(S)$, hereafter also denoted BME($n$) or
$\mathcal{P}_n$ as in \cite{Eickmeyer} and \cite{Rudy}. The vertices
of $\mathcal{P}_n$ are precisely the $(2n-5)!!$ vectors
$\mathbf{c}(t).$ Minimizing our dot product over this polytope is
equivalent to minimizing over the vertices, and thus amenable to the
simplex method.

In Figure~\ref{f:2dbmes} we see the 2-dimensional polytope
$\mathcal{P}_4.$ In that figure we illustrate a simplifying choice
that will be used throughout: rather than the original fractional
coordinates $c_{ij}$ we will scale by a factor of $2^{n-2},$ giving
coordinates $x_{ij}=2^{n-2}c_{ij} = 2^{n-2-l}.$ Since the furthest
apart any two leaves may be is a distance of $n-2$ internal nodes,
this scaling will result in integral coordinates.

\begin{figure}[b!]\centering
                  \includegraphics[width=\textwidth]{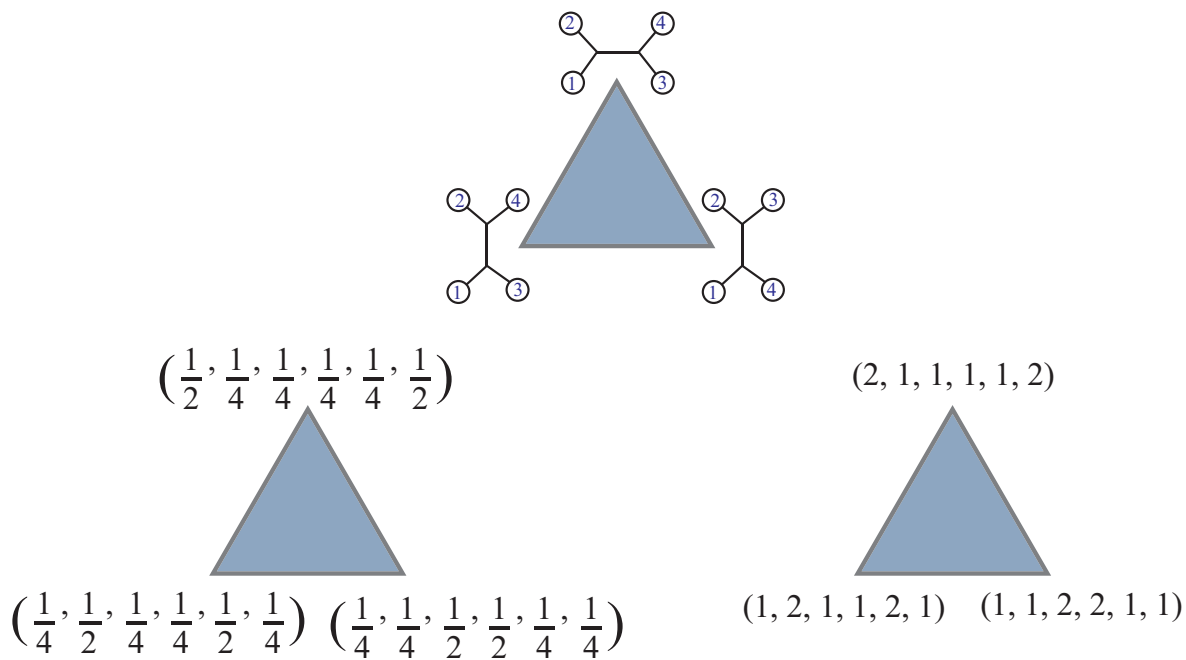}
\caption{The polytope $\mathcal{P}_4$ is a triangle. At the top we
label the vertices with the three binary trees with leaves $1\dots
4$. Each edge shows a nearest-neighbor interchange; for instance the
exchange of leaves 3 and 4 on the bottom edge. At bottom left are
Pauplin's original coordinates and at bottom right are the
coordinates, scaled by $2^{n-2}=4$, which we will
use.}\label{f:2dbmes}
\end{figure}

A \emph{clade} is a subgraph of a binary tree induced by an internal
 (degree three) node and all of the leaves descended from it in a particular
 direction. In other words: given an
 internal node $v$ we choose two of its edges and all of the leaves that
 are connected to $v$ via those two edges. Equivalently, given any
 internal edge, its deletion separates the tree into two clades. Two clades on the same tree
  must be either disjoint or \emph{nested}, one contained in the other. A
 \emph{cherry} is a clade with two leaves. We often refer to a clade
 by its set of (2 or more) leaves.  A pair of \emph{intersecting cherries} $\{a,b\}$
 and $\{b,c\}$ have intersection in one leaf $b$, and thus cannot exist both on the same tree. A \emph{caterpillar} is a tree
 with only two cherries.

\section{New Results}

A \emph{facet} of a polytope is a top-dimensional face of that
polytope's boundary, or a co-dimension-1 face. Faces of a polytope
can be of any dimension, from 0 to that of the (improper) face which
is the polytope itself. Our main results are to describe many new
faces, especially facets, of the $n^{th}$ balanced minimal evolution
polytope $\mathcal{P}_n$.

 For $n=5$ we completely classify the facets according to
combinatorial type of their vertices. There are three classes of
facet for $n=5$, which we refer to as \emph{intersecting-cherry
facets, caterpillar facets}, and \emph{cyclic-ordering facets}.
There are respectively 30, 10 and 12 of these types of facet in
$\mathcal{P}_5.$

In Theorem~\ref{th:gen_cherry} we show that any pair of intersecting
cherries corresponds to a facet of $\mathcal{P}_n.$ In
Theorem~\ref{th:cherry} we (redundantly, for demonstrative purpose)
show a special case of this facet for $n=5$ and in
Theorem~\ref{th:bcher} we show that for $n=5$ these facets turn out
to be equivalent to Birkhoff polytopes.

In Theorem~\ref{th:cat} we show that any caterpillar tree with fixed
ends corresponds to a facet of $\mathcal{P}_n.$ For $n=5$ we show in
Theorem~\ref{th:bcat} that this facet is a Birkhoff polytope.

In Theorem~\ref{th:neck} we show that, for $n=5,$ for each free
cyclic ordering of leaves there is a corresponding facet which is
combinatorially equivalent to a simplex.

 The right half of
Table~\ref{facts} summarizes these new results.
\begin{table}[hb!]
\begin{tabular}{|c|c|c|c||c|c|c|c|}
\hline
$n$ & dim. & vertices &  facets & facet inequalities & number of & number of \\
&&&&(classification)& facets &   vertices \\
&&&&&& in facet\\
\hline \hline
3 & 0 & 1 & 0 &-&-&- \\
\hline 4 & 2 & 3 & 3 & $ x_{ab}\ge 1$ & 3 & 2 \\   \cline{5-7}
&&&& $x_{ab}+x_{bc}-x_{ac} \le 2$ & 3&2\\
\hline
5 & 5 & 15 & 52 &  $x_{ab}\ge 1$ & 10&6 \\
&&&& (caterpillar)&&\\
 \cline{5-7}
&&&&$ x_{ab}+x_{bc}-x_{ac} \le 4$ & 30 & 6\\
&&&& (intersecting-cherry)&&\\
\cline{5-7} &&&&\scriptsize{ $x_{ab}+x_{bc}+x_{cd}+x_{df}+x_{fa}\le
13$} & 12 & 5\\
&&&& (cyclic ordering)&&\\
 \hline
 6 & 9 & 105 & 90262 &  $x_{ab}\ge 1$ & 15 &
24\\
&&&& (caterpillar)&&\\
  \cline{5-7}
 &&&& $~ x_{ab}+x_{bc}-x_{ac} \le 8$ &$60$ & $30$ \\
 &&&& (intersecting-cherry)&&\\
  \hline \hline
 \rule{0pt}{2.6ex}\rule[-1.2ex]{0pt}{0pt}   $n>4$ & ${n \choose 2}-n$& $(2n-5)!!$ & ? & $x_{ab}\ge 1$ &${n \choose
2}$& $(n-2)!$\\
&&&& (caterpillar)&&\\
  \cline{5-7}
 \rule{0pt}{2.6ex}\rule[-1.2ex]{0pt}{0pt}  &&&&{\scriptsize $~x_{ab}+x_{bc}-x_{ac} \le 2^{n-3}$} &${n \choose 2}(n-2)$ & $2(2n-7)!!$ \\
 &&&& (intersecting-cherry)&&\\
  \hline
\end{tabular}\caption{Stats for the BME polytopes $\mathcal{P}_n$. The first four columns are found in \cite{Huggins} and \cite{Rudy}. The inequalities are given for any $a,b,c,\dots \in [n].$ Each can be translated to an inequality in the coordinates
 $c_{ij}$ simply by dividing the right hand side by $2^{n-2}.$ For instance, when $n=4$, the second inequality becomes $c_{ab}+c_{bc}-c_{ac} \le 1/2.$ Note that for
 $n=4$ the three facets are described twice: our inequalities are redundant.\label{facts}}
\end{table}

First though, in the next section, we go over some previously
discovered facts about the edges and faces of the BME polytopes. Our
contribution there is Theorem~\ref{th:noclade}, in which we show
that clade-faces can never be facets. We also take the opportunity
to advertise future directions for the research.

\section{Edges, Clade-faces and future goals}
%%%%%%%%%%%%%%%%%%%%%%%%%%%%%%%%%%%
%
% Review of Rudy's results
% clade faces, edges, known f-vectors
%%%%%%%%%%%%%%%%%%%%%%%%%%%%%%%%%%%

Known results about the BME
 polytope are closely related to several algorithms used to
 determine optimal phylogenetic trees. Of course with a
 reasonably small set of species or individuals one could simply
 create the entire (finite) space of
 all the possible binary trees $t$ with those species as the leaves,
 calculating the dot product $\mathbf{d}\cdot\mathbf{c}(t)$ for each one
 and then choosing the optimal tree as the one minimizing this
 product. Since this procedure would take far too long (it is NP-hard, as pointed out in \cite{Day} and \cite{Fiorini}) as soon as
 the size of the set grows beyond a certain point, we are interested
 in shortcut approaches. Two of these are the fastME algorithm and
 the neighbor joining algorithm. The former is introduced in
 \cite{fastme} and the latter is developed in \cite{Saitou}.

In \cite{Gascuel} the authors show that neighbor-joining is a greedy
algorithm for the BME method. The fastME algorithm however
 operates by searching the space of binary trees, moving from one to
 another via \emph{nearest-neighbor interchange} moves. These moves
 are illustrated by the edges of the triangle in
 Figure~\ref{f:2dbmes}. Thus one goal for further study of the BME polytope is a more complete
 description of its edges, in order to more fully realize the simplex method. In \cite{Rudy} the authors show that any
 subtree-prune-regraft move is associated to an edge in the BME
 polytope. The study of the facets of the BME polytope which we
 begin here can be seen as an alternate path to hopefully even better approximations of the simplex method.

Since the total number of facets grows so quickly (90262 for $n=6$ and beyond
our computational patience for $n=7$) and since the problem is $NP$-hard, we
doubt that a complete description of facets will be easy to find. Even if it
was found the simplex method on all these facets may be infeasible. In the
current work our stated desire to completely characterize the face structure of
the BME polytope must be taken in this light: any advances are valuable despite
the fact that we may be on an endless journey.  The value of this knowledge is
in its potential application, via the following conjecture: there is a subset
of facet inequalities of the BME polytope (as found in this paper and its
sequel) which will give us a useful relaxation of the BME polytope.

In fact we conjecture that with just a fraction of the list of facet
 and face inequalities of the BME polytope we can describe a larger,
enveloping polytope which recovers most or all of our original integral BME
vertices, plus additional vertices with detectably incorrect coordinates.
Should the conjecture hold, the inequalities we use for the relaxation could be
generated as needed in a branch-and-bound algorithm, halting when one of our
powers-of-2 vertices is returned. There is possible potential for gains in
speed over the existing algorithms, subject of course to testing.

In the sequel to this paper we describe facets and faces based on trees that
display a given split. A \emph{split} of the set $S$ of leaves for our BME
trees is a partition of $S$ into two non-empty parts, $S_1$ and $S_2$. A tree
\emph{displays} a split if $S_1$ makes up the leaves of a \emph{clade}. ($S_2$
will make up the leaves of another clade.)

In \cite{Rudy} it is
 proven that any set of disjoint clades is associated with a
 specific face of the BME polytope. The clade-faces turn out to be combinatorially equivalent to
smaller-dimensional BME polytopes. Precisely, given a collection of
$k$ clades using disjoint subsets of $S$ as leaves, the face of
$\mathcal{P}_n$ corresponding to this clade will be combinatorially
equivalent to $\mathcal{P}_{n-y+k}$ where $y$ is the total number of
leaves in the $k$ clades. That is, this face will be itself a BME
polytope, equivalent to one based on a set of $n-y+k$ leaves. The
$k$ clades play the role of leaves, since they are fixed. Any vertex
of this face can be described as a binary tree with $n$ leaves such
that all $k$ disjoint clades are present. However, these clade-faces
fail to describe any of the facets of the BME polytope.

%%%%%%%%%%%%%%%%%%%%%%%%%%%%%%%%%%
%
% No clade face is a facet
%
%%%%%%%%%%%%%%%%%%%%%%%%%%%%%%%%%%

\begin{theorem}\label{th:noclade}
If $n\ge4$ then no clade face of $\mathcal{P}_n$ is a facet of
$\mathcal{P}_n$.
\end{theorem}
We expect this to be true since the largest dimension clade-face
would be that associated to a single cherry: the smallest clade.
Here is a proof that takes a more general approach.
\begin{proof}
 Since a face of $\mathcal{P}_n$
corresponding to a disjoint set of $k$ clades containing a total of
$y$ leaves is combinatorially equivalent itself to a smaller BME
polytope, its dimension is that of the polytope
$\mathcal{P}_{n-y+k}$.  Now,
 a facet of a BME polytope has dimension ${n \choose 2} -n -1$, for
 $n$ leaves.  Thus if a facet was described by a disjoint set of $k$
 clades
 containing a total of $y$ leaves, we could say that

$${n \choose 2} -n -1 = {n-y+k  \choose 2} -(n-y+k).$$

 This equation implies the quadratic equation $p^2 + (2n-3)p +2 = 0$ , where  $p = k-y$ must be a negative
 integer. The roots occur at  $p= -n+ \sqrt{q}\slash 2 + 3\slash 2$  where $q =
 4n^2-12n+1.$

 So for $p$ to be an integer, we need $\sqrt{q}$ to be an odd integer, so
 $q$
 is the square of an odd (positive) integer $2m-1$.

 Thus $4n^2-12n+1 = (2m-1)^2 = 4m^2 -4m  +1$     for  integer $m>0.$

 Subtracting the 1's and dividing by 4 we get:
 $$n(n-3) = m(m-1).$$

Letting $n=i+1$ for $i>1$ we see that $n(n-3) = i(i-1) - 2.$  Thus
any term $a_m$ in the sequence of integers $m(m-1)$ for $m>1$ will
always be equal to $n(n-3) +2=b_n+2$ for some $n$. Since $n(n-3)$
increases faster than by simply adding 2, the term $a_m$ in question
cannot be equal to any term after $b_n$ (nor any before, since both
sequences are increasing.)

In fact the only time that the equation can hold is for $m=1$ and
$n=3.$
\end{proof}

This negative fact of course raises the question of how to characterize and
describe the facets of $\mathcal{P}_n$. We would eventually like a complete
description, both combinatorially and geometrically. On the combinatorial side
we would like to know which sets of vertices are those of a facet, and what
other polytopes and constructions of polytopes (products, sums, pyramids,
polars) those facets are equivalent to. On the geometrical side we would like
to know how to quickly find the list of facet inequalities that describe
$\mathcal{P}_n$. In Figure~\ref{f:bme5data} we show the data for $n=5.$

\begin{figure}[b]\centering
                  \includegraphics[width=\textwidth]{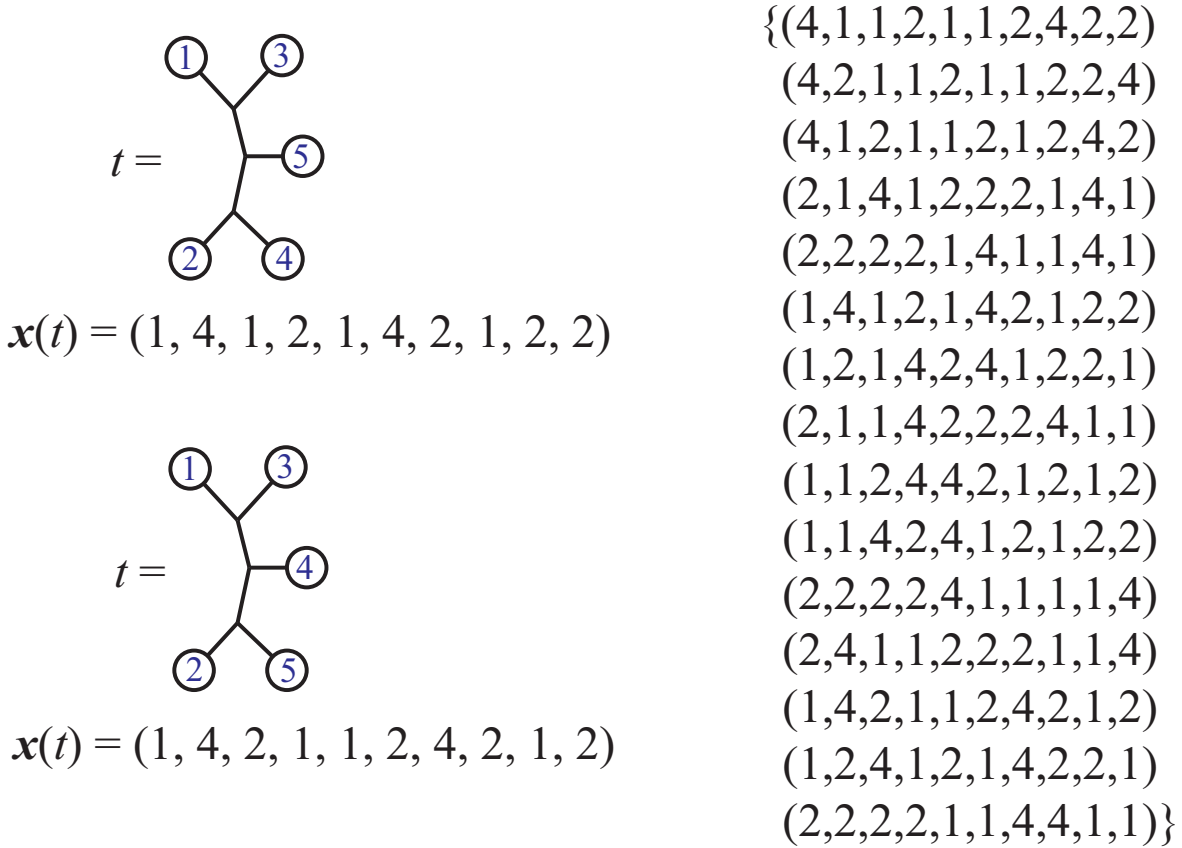}
\caption{Two sample vertex trees of $\mathcal{P}_5$ with their
respective coordinates shown beneath, followed by all 15 vertex
points calculated for n=5, and the $f$-vector for $\mathcal{P}_5$
(which gives the number of faces in each dimension starting with 0)
as found by polymake \cite{polymake}.} \label{f:bme5data}
\end{figure}

%%%%%%%%%%%%%%%%%%%%%%%%%%%%%%%%%%
%
% Type 1 facets
%
%%%%%%%%%%%%%%%%%%%%%%%%%%%%%%%%%%

\section{Facets from intersecting cherries.}

 The first type of facet
of $\mathcal{P}_5$ that we found is associated to any pair of
elements of $S = [5],$ along with a third element chosen after the
pair. Thus there are ${5\choose 2}(3) =30$ of these facets. Each of
these facets has its set of vertices as follows:

\begin{theorem}\label{th:cherry}
For each pair of cherries with leaves $\{a,b\}$ and $\{b,c\},$ where the pair
$a,c$ and the element of intersection $b$  are three distinct elements from
$S=[5],$ there is a facet of $\mathcal{P}_5$ whose six vertices correspond to
trees that have one of the two cherries.
\end{theorem}

Figure~\ref{f:intersect_facet_abc} shows the geometry of an
\emph{intersecting-cherry facet} of $\mathcal{P}_5$.

\begin{figure}[b]\centering
                  \includegraphics[width=\textwidth]{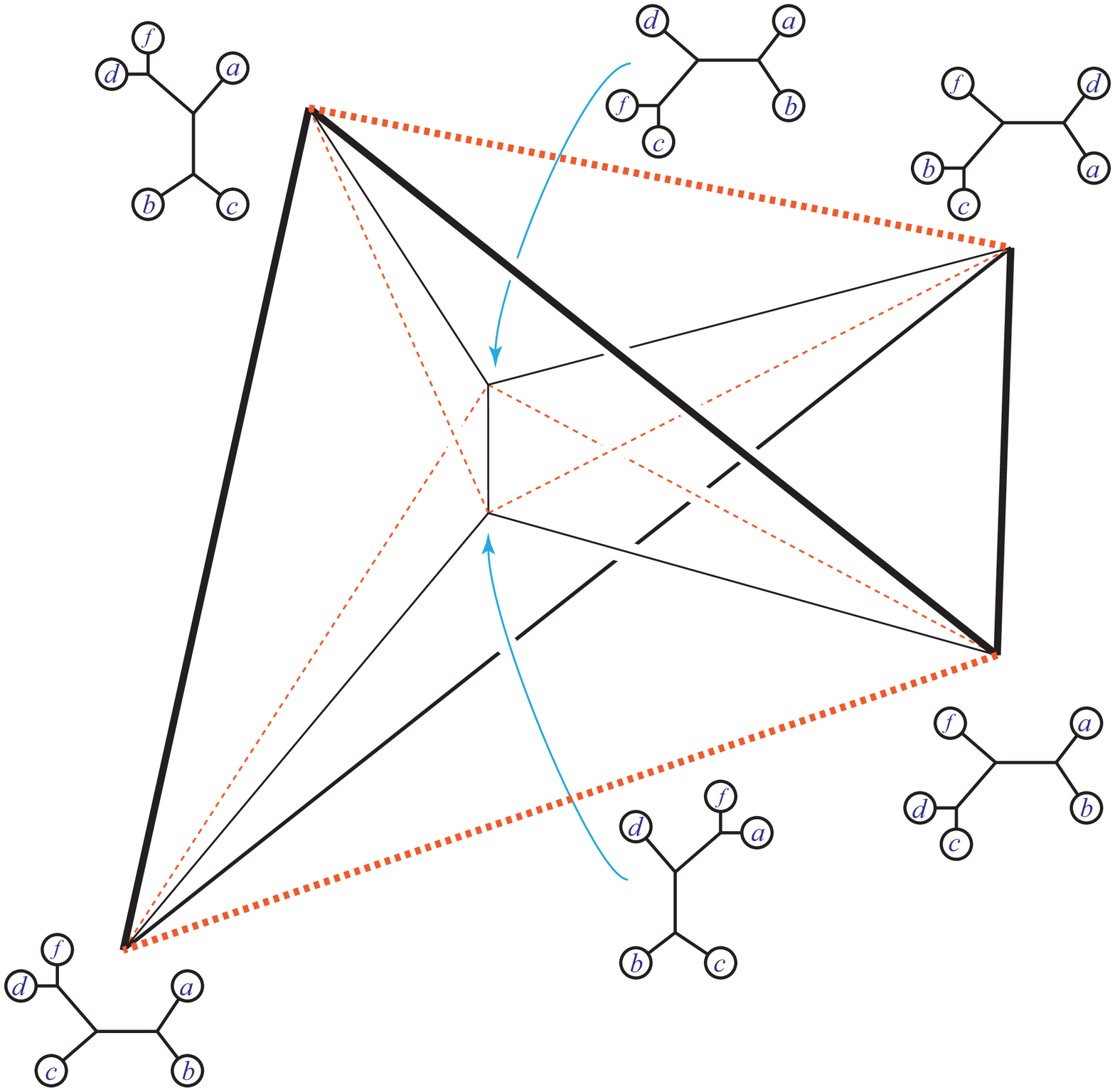}
\caption{A generic facet of $\mathcal{P}_5$ with each vertex labeled
by a tree which contains one of two intersecting cherries:
 $\{a,b\}$ and $\{b,c\}$. The dashed edges outline the clade-faces (triangles) associated with those two cherries.}\label{f:intersect_facet_abc}
\end{figure}

\begin{proof}
There are six total vertices since given one of the pair of cherries
there are 3 trees which have that cherry, since there are 3 elements
to choose from to make the lone leaf. To show that these six
vertices are the vertices of a face we need to find a linear
inequality satisfied by all the vertices of $\mathcal{P}_5$ which
becomes an equality only for the specified six vertices. Then to
show that the face is a facet we need to show that its dimension is
one less than the dimension of the entire polytope $\mathcal{P}_5$.
First we show that the trees which have either a cherry with leaves
$\{a,b\}$ or with leaves $\{b,c\}$ have associated points obeying:
$$x_{ab}+x_{bc}-x_{ac}=4.$$

This equation holds for our trees since if $\{a,b\}$ is the cherry
then $x_{ab}=4$ and $x_{ac} = x_{bc}.$ Likewise if $\{b,c\}$ is the
cherry then $x_{bc}=4$ and $x_{ac} = x_{ab}.$

Now we need to show that for any vertex that has neither of our pair
of cherries, then that vertex satisfies:
$$x_{ab}+x_{bc}-x_{ac}<4.$$

This inequality holds since having neither cherry with leaves
$\{a,b\}$ nor with leaves $\{b,c\}$ implies that $x_{ab} \le 2$ and
$x_{bc} \le 2$, while we know that $x_{ac}\ge 1$.

To see that our face is a 4-dimensional facet, we show that it
contains a \emph{flag} of subfaces (sequence of faces each contained
in its successor) which is of length 5. We can proceed starting with
any vertex and edge, since when $n=4$ any pair of vertices have an
edge between them. Our flag chosen for the purposes of this proof is
shown in Figure~\ref{f:intersect_facet_flag}, left to right with the
vertex and edge first. We choose any vertex, but then choose an edge
which connects that vertex with another that shares with the first
one of our special cherries, say $\{b,c\}.$

Next, the dimension 2 subface in our flag is formed by adding the
third vertex that also contains the cherry $\{b,c\}.$ These three
vertices form a clade face--the clade is the cherry.

The dimension 3 subface is found by adding a fourth vertex whose
tree has both cherries $\{a,b\}$ and $\{c,f\}$. Together these four
make a face: all four points obey the equation $x_{bd} - x_{cd} =0.$
The last two remaining trees in the facet are forced to obey $x_{bd}
- x_{cd} <0.$

\begin{figure}[b]\centering
                  \includegraphics[width=\textwidth]{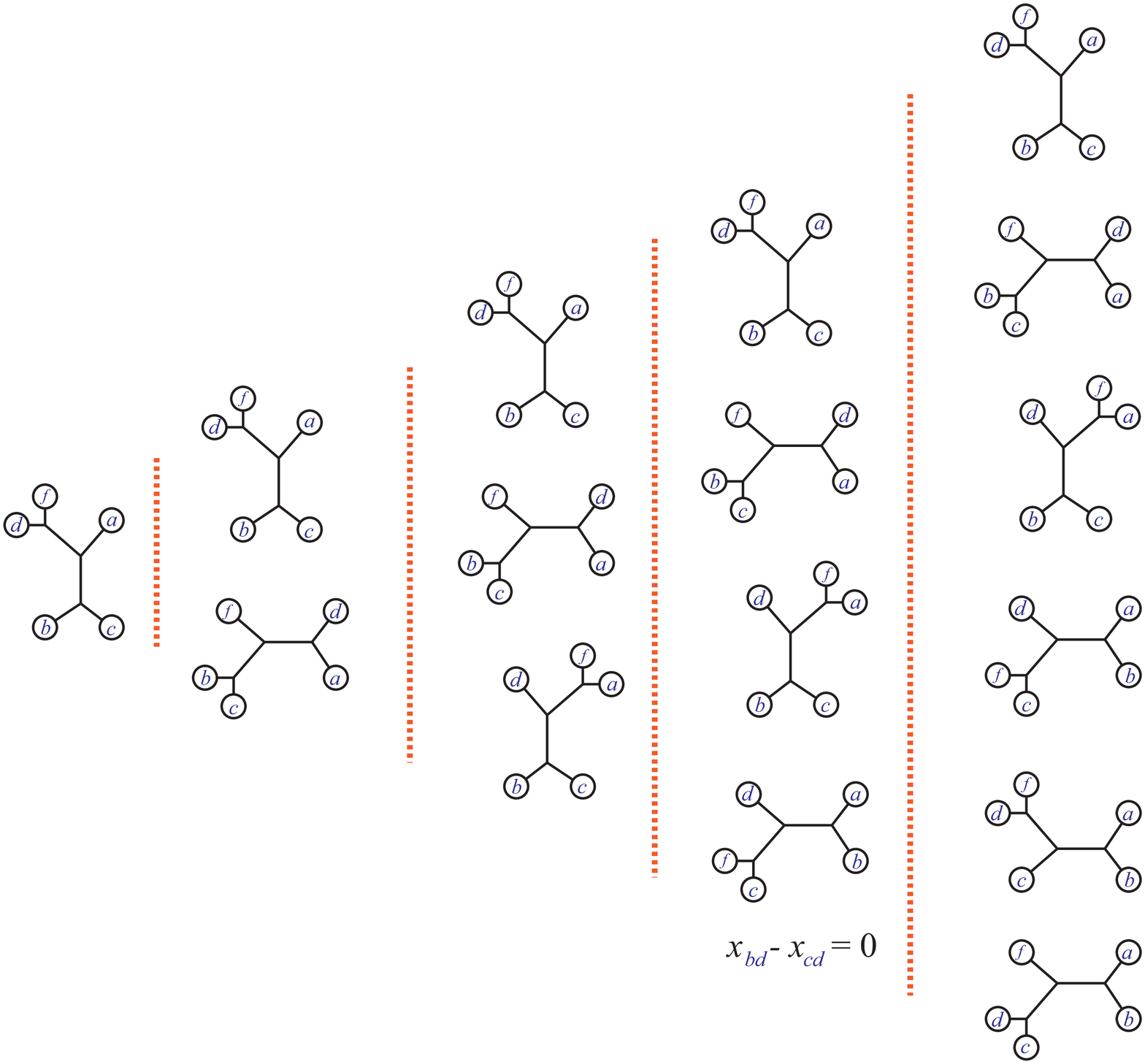}
\caption{Left to right these columns show the sets of trees in faces
that form a flag of a facet of $\mathcal{P}_5$ based on  two
intersecting cherries.}\label{f:intersect_facet_flag}
\end{figure}

\end{proof}

Consider an $n \times n$ matrix as a vector with $n^2$ components.
Taking the convex hull of the $n!$ permutation matrices gives a
polytope known as $B_n$, or $B(n)$, the \emph{Birkhoff} polytope, or
assignment polytope, of order $n$. This polytope has dimension
$(n-1)^2$, and appears in many situations, as seen in
\cite{Billera}. Here it appears again:

\begin{theorem}\label{th:bcher}
The intersecting-cherry facets of $\mathcal{P}_5$ are
combinatorially equivalent to the Birkhoff polytope of dimension 4.
\end{theorem}

\begin{proof}
This fact was first verified by polymake for a specific
intersecting-cherry facet, where the isomorphism of vertices can be
seen as preserving vertex-facet incidence. Since all the
intersecting-cherry facets are combinatorially equivalent by a
common permutation of the coordinates, checking one is sufficient
for checking all. To see the isomorphism compare the two Schlegel
diagrams: from Figure~\ref{f:intersect_facet_abc} and from
Figure~\ref{f:birkhoff-iso}.

%\begin{figure}[b]\centering
%                  \includegraphics[width=4.5in]{birkhoff-inter.eps}
%\caption{On the top is the Birkhoff polytope $B(3)$ with vertices
%labeled by the permutation matrices. On the bottom is a facet of
%$\mathcal{P}_5$ with each vertex labeled by the tree corresponding
%to the permutation matrix in the corresponding
%position.}\label{f:birkhoff-inter}
%\end{figure}
\end{proof}

We leave open for future study the question of how one might
determine the general isomorphism between an intersecting-cherry
facet and the 4d-Birkhoff polytope. It would also be quite
interesting to see if the relationship extends to higher dimensions.
For now we only say the following for the general case of
$\mathcal{P}_n$:

\begin{theorem}\label{th:gen_cherry}
 Each pair of intersecting cherries from $S =[n]$
corresponds to a certain facet of the polytope $\mathcal{P}_n$.
\end{theorem}
\begin{proof}
The trees which have either a cherry with leaves $\{a,b\}$ or with
leaves $\{b,c\}$ have associated points obeying:
$$x_{ab}+x_{bc}-x_{ac}=2^{n-3}.$$

This equation holds for our trees since if $\{a,b\}$ is the cherry
then $x_{ab}=2^{n-3}$ and $x_{ac} = x_{bc}.$ Likewise if $\{b,c\}$
is the cherry then $x_{bc}=2^{n-3}$ and $x_{ac} = x_{ab}.$

For any vertex that has neither of our pair of cherries, then that
vertex satisfies:
$$x_{ab}+x_{bc}-x_{ac}<2^{n-3}.$$

This inequality holds since having neither cherry with leaves
$\{a,b\}$ nor with leaves $\{b,c\}$ implies that $x_{ab} \le
2^{n-4}$ and $x_{bc} \le 2^{n-4}$, while we know that $x_{ac}\ge 1$.

So far we have shown that the collection of trees with either cherry
$\{a,b\}$ or $\{b,c\}$ forms a face $P_{abc}$ with inequality
$x_{ab}+x_{bc}-x_{ac}\le2^{n-3}$. Next we show that the face
$P_{abc}$ is in fact a facet, of dimension ${n \choose 2} - n -1.$
The strategy is to show existence of a flag of $\mathcal{P}_n$
beginning with $P_{abc}$ and ending with a single vertex, which has
total length ${n \choose 2} - n.$ We start with a chain of sub-faces
of $P_{abc}$ which has length $n-3$, including $P_{abc}$ itself.
Then we show a final sub-face which has dimension ${n-1 \choose 2}
-(n-1)$. Thus the entire flag is of length ${n-1 \choose 2}-(n-1) +
1+ n-3 = {n \choose 2}-n.$ We have illustrated this flag in
Figure~\ref{f:int_flag}.

After $P_{abc}$ the largest face in our flag is the one whose vertices are
described as each corresponding to a tree that has either the cherry $\{a,b\}$
or has both the cherry $\{b,c\}$ and the cherry $\{a,f\}$. We call this face
$P_{abc,f}.$ The vertices of $P_{abc,f}$ obey the equality:
$$x_{bc}+x_{bf}-x_{ac}-x_{af}= 0.$$ The trees with cherry $\{a,b\}$ have equal
distances from those two leaves, while the remaining trees have coordinates
$x_{bc}=x_{af}$ and $x_{bf}=x_{ac}$. Trees that have cherry $\{b,c\}$ but not
cherry $\{a,f\}$ obey the inequality: $$x_{bc}+x_{bf}-x_{ac}-x_{af}> 0.$$ This
is true since the tree cannot have the cherry $\{a,c\}$ either, so $x_{ac} +
x_{af}\le 2^{n-3}.$

Next, given an ordering $(y_1, y_2,\dots,y_{n-4})$ of the leaves not including
$a,b,c$ or $f,$ there is a sequence of $n-5$ sub-faces called $P^1_{abc,f},
P^2_{abc,f},\dots,P^{n-5}_{abc,f}.$ The face $P^k_{abc,f}$ has vertices
described as each corresponding to a tree that has either the cherry $\{a,b\}$
or is a caterpillar that has both the cherry $\{a,f\}$, at one end, and the
cherry $\{b,c\}$ at the other. Between the cherries in the caterpillar are the
leaves $y_1,\dots,y_{k}$ in that order beginning closest to the cherry
$\{a,f\}$. The remaining leaves $y_j$ for $j>k$ fill in the caterpillar in any
order; note that there are $n-5$ such collections since the last leaf is
determined when we reach $P^{n-5}_{abc,f}.$ See Figure~\ref{f:int_flag}.

The vertices of $P^k_{abc,f}$ obey the linear equality:
$$(2^{n-3}-1)x_{ay_k} - (2^{n-3}-1)x_{by_k}+ (2^{n-3-k}-2^k)x_{ab} = (2^{n-3-k}-2^k)2^{n-3}, $$
or, more conveniently,
$$(2^{n-3}-1)(x_{ay_k}-x_{by_k})= (2^{n-3} - x_{ab})(2^{n-3-k}-2^k).$$
The equality is clear for trees that have the cherry $\{a,b\}$, since it
becomes $0=0$. For caterpillar trees of the face $P^k_{abc,f}$ we have $x_{ab}
=1$ and $x_{ay_k}-x_{by_k}= 2^{n-3-k}-2^k.$

We need to show that for trees in $P_{abc,f}$ that are not in $P^1_{abc,f}$ we
have the inequality:
$$(2^{n-3}-1)(x_{ay_1}-x_{by_1})< (2^{n-3} - x_{ab})(2^{n-3-1}-2^1).$$

There are two cases:

Case 1) If the tree is a caterpillar, with leaf $y_1$ more than two nodes from
leaf $a$, then the inequality follows from $x_{ab} = 1$ and
$x_{ay_1}-x_{by_1}<2^{n-4}-2.$

Case 2) If the tree is not a caterpillar, we show the inequality by induction
on the number $n$ of leaves.  We check the base case $n=6,$ where the
inequality becomes $0 < (2^3-2)(2^2-2).$

Assuming the inequality for $m<n$, then in the case for $n$ leaves we choose a
cherry and replace it with a single leaf.
There are two subcases of Case 2.

\emph{Subcase (i)}: If leaf $y_1$ was in the chosen
cherry then we call the replacement leaf $y_1$ instead, and by induction we
have the inequality:
$$(2^{n-1-3}-1)(x_{ay_1}-x_{by_1})< (2^{n-1-3} - x_{ab}/2)(2^{n-1-4}-2).$$
where the values for the coordinates mentioned are the same as in the
$n$-leaved tree before replacement.

Multiplying by 2, we get:

 $(2^{n-3}-2)(x_{ay_1}-x_{by_1})< 2(2^{n-1-3} -
x_{ab}/2)(2^{n-5}-2)$

$= (2^{n-3} - x_{ab})(2^{n-5}-2).$

Expanding on the left and then adding to both sides gives:

$(2^{n-3}-1)(x_{ay_1}-x_{by_1})< (2^{n-3} - x_{ab})(2^{n-5}-2)
+x_{ay_1}-x_{by_1} $

Expanding on the right via $2^{n-5} = 2^{n-4} - 2^{n-5},$ we get:

$(2^{n-3}-1)(x_{ay_1}-x_{by_1})< (2^{n-3} - x_{ab})(2^{n-4}-2) -
2^{n-5}(2^{n-3} - x_{ab}) +x_{ay_1}-x_{by_1} $

$= (2^{n-3} - x_{ab})(2^{n-4}-2) - (2^{n-4})^2 +2^{n-4}x_{ab}/2)
+x_{ay_1}-x_{by_1} $

 Using the facts that, for our
non-caterpillar tree, we know $x_{ay_1}-x_{by_1} \le 2^{n-4}-2$ and $x_{ab} \le
 2^{n-4}$, we get:

$(2^{n-3}-1)(x_{ay_1}-x_{by_1})< (2^{n-3} - x_{ab})(2^{n-4}-2) -
(\frac{1}{2}(2^{n-4})^2 - 2^{n-4} +2). $

The last term is a polynomial in $2^{n-4}$ whose minimum is $3/2$ when $n=4.$ Thus it can be dropped to achieve the desired inequality.

\emph{Subcase (ii)}
If leaf $y_1$ is not in the chosen
cherry then by induction we
have the inequality:
$$(2^{n-1-3}-1)(x_{ay_1}/2-x_{by_1}/2)< (2^{n-1-3} - x_{ab}/2)(2^{n-1-4}-2).$$
where the values for the coordinates mentioned are the same as in the
$n$-leaved tree before replacement.

Multiplying by 4, we get:

 $(2^{n-3}-2)(x_{ay_1}-x_{by_1})< (2^{n-3} -
x_{ab})(2^{n-4}-4)$

Expanding on the left and then adding to both sides gives:

$(2^{n-3}-1)(x_{ay_1}-x_{by_1})< (2^{n-3} - x_{ab})(2^{n-4}-2) -2(2^{n-3} - x_{ab})
+x_{ay_1}-x_{by_1} $

 Using the facts that, for our
non-caterpillar tree, we know $x_{ay_1}-x_{by_1} \le 2^{n-4}-2$ and $x_{ab} \le
 2^{n-4}$, we get:

$(2^{n-3}-1)(x_{ay_1}-x_{by_1})< (2^{n-3} - x_{ab})(2^{n-4}-2) -
(2^{n-2}-2^{n-3} -2^{n-4}+2)$

$=(2^{n-3} - x_{ab})(2^{n-4}-2) -
(4(2^{n-4})-2(2^{n-4}) -2^{n-4}+2)$

$=(2^{n-3} - x_{ab})(2^{n-4}-2) -
(2^{n-4}+2).$

Since the last term is greater than 2, it can be dropped to achieve the desired inequality.

Next we need to check that for trees in $P^k_{abc,f}$ that are not in $P^{k+1}_{abc,f}$ we
have the inequality:
$$(2^{n-3}-1)(x_{ay_k}-x_{by_k})< (2^{n-3} - x_{ab})(2^{n-3-k}-2^k).$$

This inequality is straightforward, since $x_{ab} =1$ and since the leaf $y_k$ is forced to be closer
to leaf $b$ and further from leaf $a$, if it is not in the $k^{th}$ position.

\begin{figure}[b]\centering
                  \includegraphics[width=\textwidth]{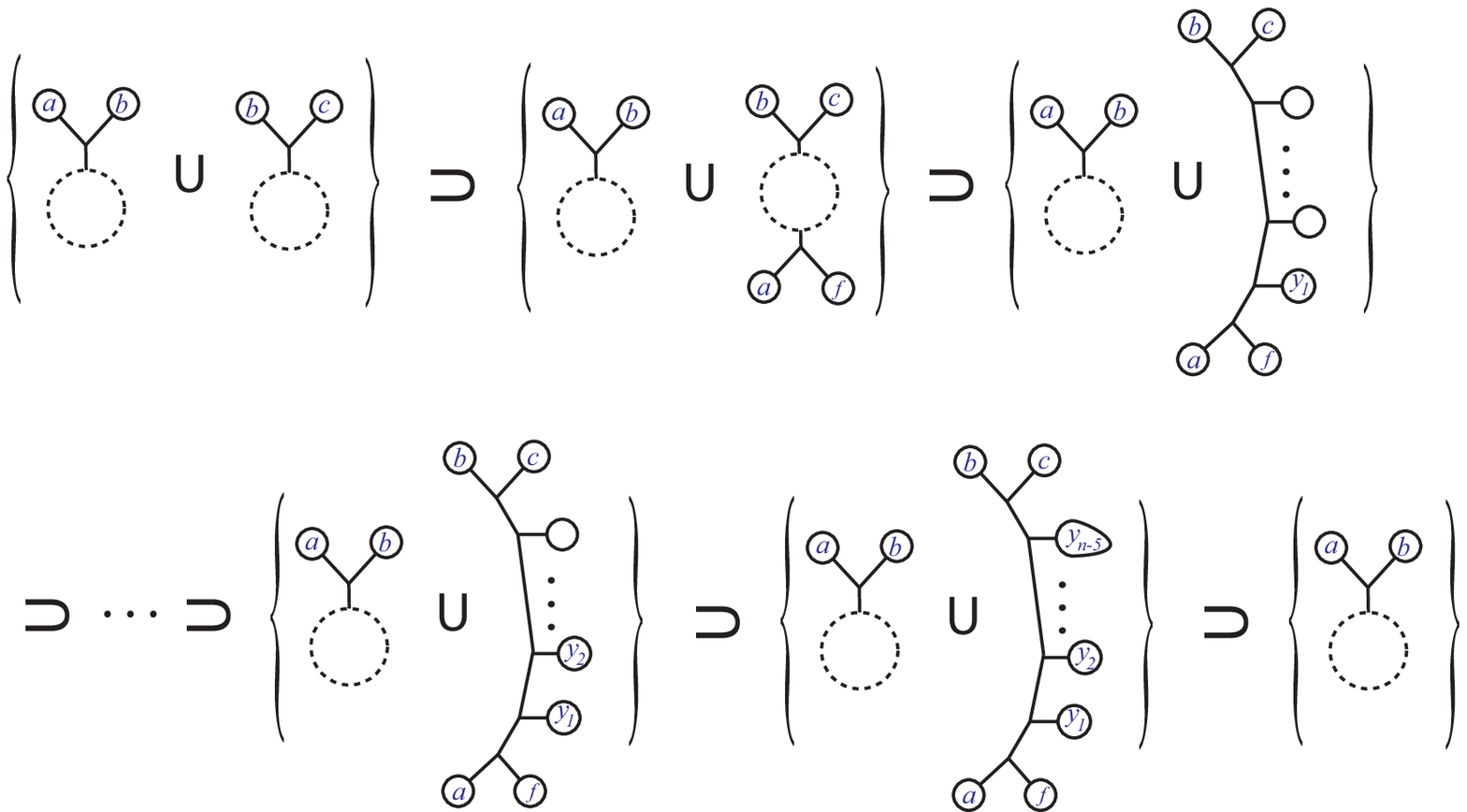}
\caption{Each set is the collection of trees (vertices of
$\mathcal{P}_n$) that have any placement of the remaining labels
from $S$ into the blank leaves shown. The containment of sets also
shows the flag of our facet $P_{abc}$. The containment is
$P_{abc}\supset P_{abc,f}\supset P^1_{abc,f}\supset
P^2_{abc,f}\supset\dots\supset P^{n-5}_{abc,f}\supset P_{ab}.$
}\label{f:int_flag}
\end{figure}
\end{proof}

The facet inequalities we describe in the above proof are equivalent to
the \emph{triangular inequalities} of Proposition 3 in \cite{Catan}.
This connection does raise the question of whether other
inequalities in that paper can lead to facets of the BME polytope.

 Note that the dimension of a general face corresponding to an
intersecting pair of cherries is greater than the dimension of the
clade-face for the clade that is one of those cherries. We
conjecture based on initial experiments (verified by polymake for
$n$=6) that in fact the dimension is ${n \choose 2}-n-1,$ implying
that these intersecting-cherry faces are indeed facets. We also
leave for the future the investigation of other sorts of
intersecting sets of clades: we conjecture that two or three or more
clades, of various sizes and tree geometry, intersecting in various
ways, will lead to further faces and facets of $\mathcal{P}_n$.
%%%%%%%%%%%%%%%%%%%%%%%%%%%%%%%%%%
%
% Type 2 facets
%
%%%%%%%%%%%%%%%%%%%%%%%%%%%%%%%%%%

\section{Facets from free cyclic orderings.}

A free circular permutation (or free cyclic ordering, or necklace)
of the elements of $S$ is only distinguished by which elements are
adjacent. It is an arrangement of the elements of $S$ around a
circle, which may be rotated or flipped. We are interested in the
binary trees on $S$ which are coplanar with a certain free cyclic
ordering on $S.$ That is, having drawn one of the two planar
versions of the cyclic ordering, we can then draw the tree in the
same plane, as in Figure~\ref{f:necklace_simplex}. The number of
trees coplanar with the free cyclic ordering is found by a simple
counting argument for $n=5:$ there are 5 choices for the first
cherry and the two for the second cherry; but then we divide by two
since the order we choose the cherries in is irrelevant, giving us
five total trees.

\begin{theorem}\label{th:neck}
For each free cyclic ordering $\mathcal{N}$ on $S$ with $|S| = 5$
there is a facet of $\mathcal{P}_5$ that is equivalent to a
4-simplex, whose five vertices correspond to trees that are coplanar
with the free cyclic ordering.
\end{theorem}
\begin{proof}
The trees which are coplanar with $\mathcal{N} = (a,b,c,d,f)$
satisfy $x_{ab}+x_{bc}+x_{cd}+x_{df}+x_{fa}= 13.$ That is because
two cherries are represented by those components, and the remaining
components are assigned values 2, 2 and 1 respectively.  All other
trees in $\mathcal{P}_5$, not coplanar with $\mathcal{N}$, obey
$x_{ab}+x_{bc}+x_{cd}+x_{df}+x_{fa}< 13.$ That is because at most
one cherry can be among those components, and the rest of the
components then can at most be assigned the value 2. Thus the
components add up to at most 12.

Thus our five trees constitute the vertices of a face of
$\mathcal{P}_5$. We show that this is a facet, with dimension equal
to 4, by establishing within it a flag of length 5. We can proceed
starting with any vertex and edge, since any pair of vertices have
an edge between them. Our flag chosen for the purposes of this proof
is shown in Figure~\ref{f:necklace_facet_flag}, left to right with
the vertex and edge first. The set of three vertices makes a
triangular face of the facet since they obey the equality $x_{bf}=1$
while the other two vertices have $x_{bf}>1.$ The set of four
vertices make a 3-simplex since they all obey the equality
$x_{bf}+x_{bd}+x_{ac} = 4,$ while the final vertex has
$x_{bf}+x_{bd}+x_{ac} >4.$

\begin{figure}[b]\centering
                  \includegraphics[width=\textwidth]{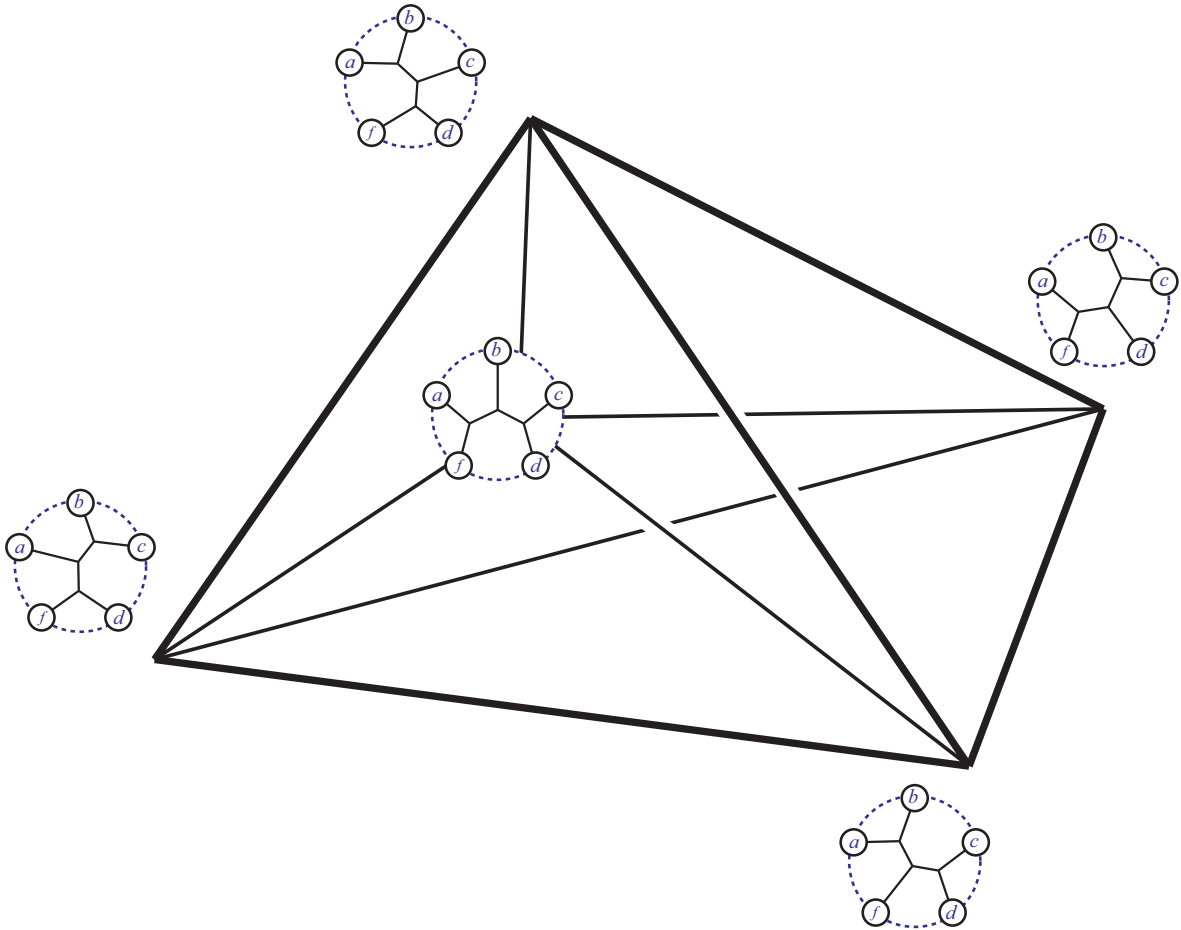}
\caption{A generic facet of $\mathcal{P}_5$ based on a free cyclic
ordering $(a,b,c,d,f)$.}\label{f:necklace_simplex}
\end{figure}

\begin{figure}[b]\centering
                  \includegraphics[width=\textwidth]{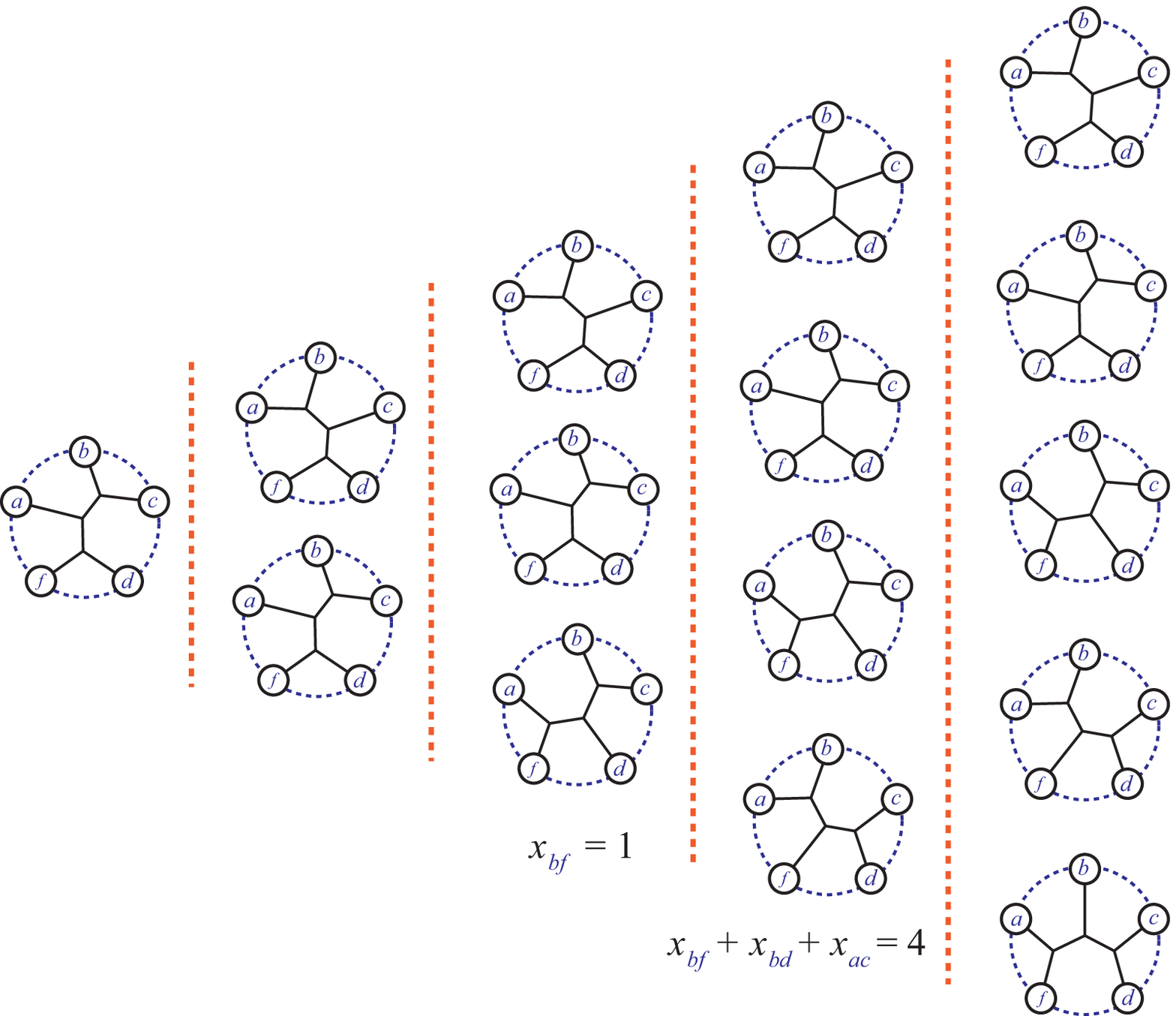}
\caption{Left to right these columns show the sets of trees in faces
that form a flag of a facet of $\mathcal{P}_5$ based on a free
cyclic ordering $(a,b,c,d,f)$.}\label{f:necklace_facet_flag}
\end{figure}
\end{proof}

The number of free cyclic orderings on 5 objects is $4!\slash 2 =
12.$ Via polymake we see that there are exactly 12 facets of
$\mathcal{P}_5$ that have 5 vertices. Thus we have accounted for all
of these facets with free cyclic orderings.

%For $S$ with $|S|>5$ our conjectural necklace-faces become much more
%complicated, since the number of cherries may vary between trees on
%the same necklace.

%%%%%%%%%%%%%%%%%%%%%%%%%%%%%%%%%%
%
% Type 3 facets
%
%%%%%%%%%%%%%%%%%%%%%%%%%%%%%%%%%%

\section{Facets from Caterpillars.}

The third type of facet for $\mathcal{P}_5$ corresponds to a choice
of two elements of $S.$ These are placed as leaves on a tree that
are as far apart as possible: in this case a distance of 3 internal
nodes on a binary caterpillar. Thus there are six ways to place the
remaining three elements of $S$ as the other three leaves, and the
\emph{caterpillar facet} has 6 vertices.

\begin{theorem}
Each pair of elements ${a,b}$ from $S$ with $|S|=5$ determines a
facet of $\mathcal{P}_5$ whose vertices are trees that have  $a$ and
$b$ as leaves of distinct cherries.
\end{theorem}

\begin{proof} The result follows from the general Theorem~\ref{th:cat} which establishes the fact for all
dimensions. Specifically, each tree that has the elements $a$ and
$b$ separated by 3 internal nodes has corresponding vector that
obeys $x_{ab} =1.$ All other trees, which do not have this property,
obey $x_{ab}>1.$ The flag of length 5 which establishes that the
face in question is indeed a facet is described inductively in the
proof of Theorem~\ref{th:cat}.
\end{proof}

The number of these facets in $\mathcal{P}_5$ is ${5 \choose 2}
=10.$ Note that now we have described $30+12+10=52$ facets of
$\mathcal{P}_5,$ the total number predicted by polymake. The three
classes of facets do not intersect for $n=5$: that is, a caterpillar
facet cannot be an intersecting cherry facet (nor vice-versa), since
the collection of caterpillar trees is determined by choosing two
leaves to be as far apart as possible, while the intersecting cherry
trees allow any two leaves to be closer than the maximum distance.

\begin{theorem}\label{th:bcat}
The caterpillar facets of $\mathcal{P}_5$ are combinatorially
equivalent to the Birkhoff polytope of dimension 4.
\end{theorem}

\begin{proof}
This fact was first verified by polymake, where the isomorphism of
vertices may be seen as preserving vertex-facet incidence. Since all
the caterpillar facets are combinatorially equivalent by a common
permutation of the coordinates, checking one is sufficient for
checking all. We illustrate the isomorphism by showing the two
Schlegel diagrams in Figure~\ref{f:birkhoff-iso}.
\end{proof}

\begin{figure}[b]\centering
                  \includegraphics[width=4.5in]{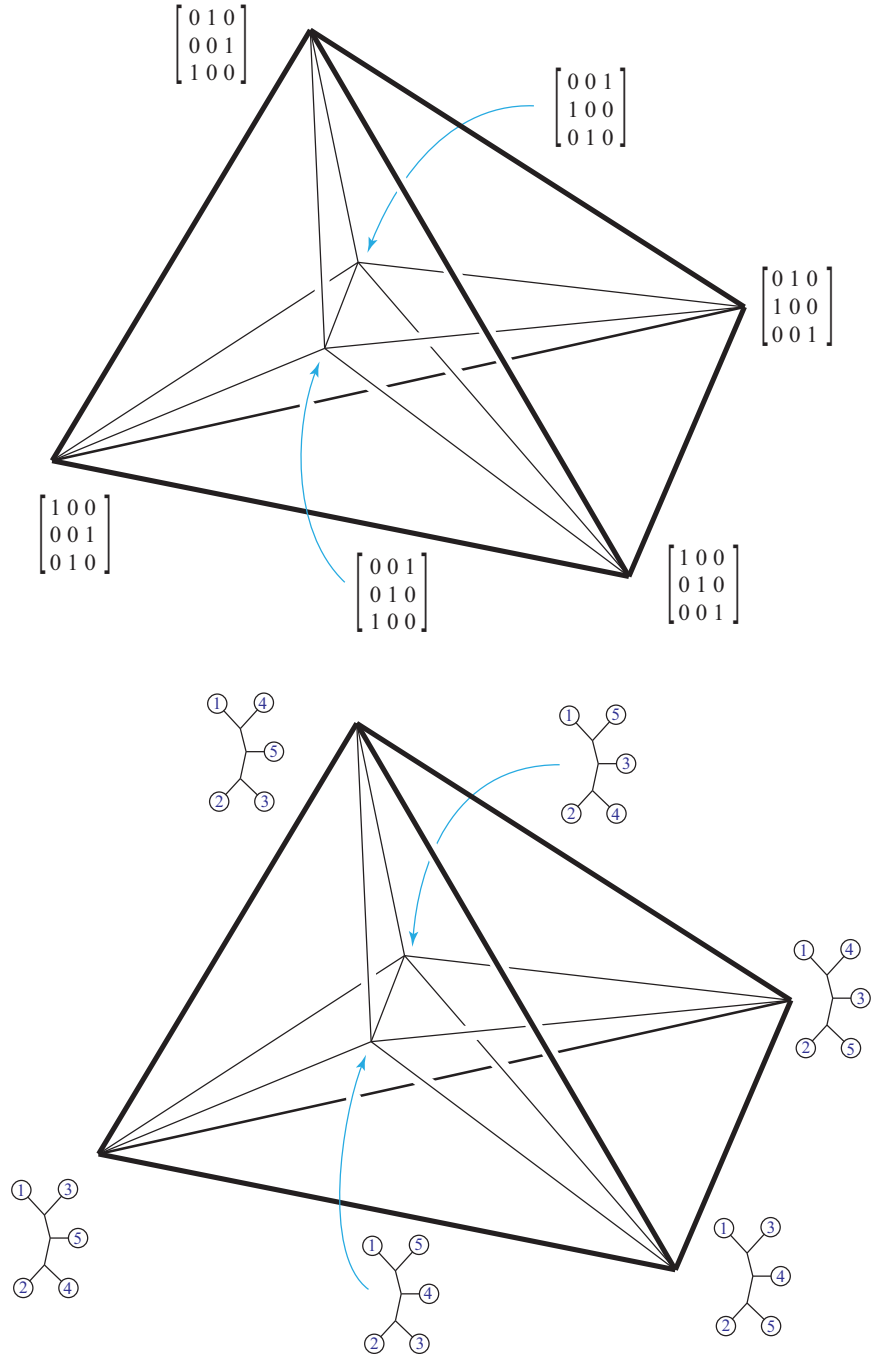}
\caption{On the top is the Birkhoff polytope $B(3)$ with vertices
labeled by the permutation matrices. On the bottom is a facet of
$\mathcal{P}_5$ with each vertex labeled by the tree corresponding
to the permutation matrix in the corresponding
position.}\label{f:birkhoff-iso}
\end{figure}

The generalization of this caterpillar facet for any size set $S$
has vertices any collection of trees with leaves $S$ that are all
binary caterpillars, with a pair of chosen species as the two which
must reside as far apart as possible--as leaves of the only two
distinct cherries. These faces are indeed facets of $\mathcal{P}_n,$
each with $(n-2)!$ vertices. In general they are not equivalent to
the Birkhoff polytope $B(n-2)$, since for $n>5$ the facets of
$\mathcal{P}_n$ have a dimension ${n \choose 2} -n-1$ which is
greater than $(n-3)^2$ (the dimension of $B(n-2).$)  However an
interesting projection is suggested by the 4-dimensional case in
Figure~\ref{f:birkhoff-iso}, where a permutation matrix
corresponding to permutation $\sigma$ is mapped to the tree with
leaves 3,4, and 5 in the order $\sigma(3), \sigma(4), \sigma(5)$. We
leave as an open question, for instance, whether this map in the
general case gives rise to a cellular projection to the Birkhoff
polytope from facets of the balanced minimal evolution polytope.

\begin{theorem}\label{th:cat}
Consider the set of binary caterpillar trees on $n$ leaves $S$ with
a given pair from the set $S$ as maximally separated leaves. The
vertices of $\mathcal{P}_n$ calculated from these caterpillar trees
are the vertices of a facet of $\mathcal{P}_n$.
\end{theorem}

Note that the number of these facets in $\mathcal{P}_n$ is ${n
\choose 2}$ for $n>4.$ For $n=4$ there are half that many, the three
edges of the triangle in Figure~\ref{f:2dbmes}, since having chosen
two elements of $S$ to be placed in the distinct cherries we
automatically determine the other two elements which will also be
placed in distinct cherries: in the notation of the proof that
follows we have for instance that $P^4_{12} = P^4_{34}.$

\begin{proof} of Theorem~\ref{th:cat}
For the purposes of this proof we choose the set $S = [n]$ and
without loss of generality we let the two fixed leaves with maximal
distance $n-2$ between them be the leaves labeled 1 and 2. We'll
continue by choosing leaves in counting order: this will be without
loss of generality since any other selection of leaves is covered by
choosing an appropriate ordering.

Thus the caterpillar trees with fixed leaves 1 and 2 obey $x_{12} =
1$, and all other points in $\mathcal{P}_n$ obey $x_{12}>1.$  We
call this face $P^n_{12}.$ Next we use induction on $n$, the number
of leaves, to show that the face $P^n_{12}$ is in fact a facet, of
dimension ${n \choose 2} - n -1.$ The strategy is to show existence
of a flag of $\mathcal{P}_n$ beginning with $P^n_{12}$ and ending
with a single vertex, which has total length ${n \choose 2} - n.$ We
start with a chain of sub-faces of $P^n_{12}$ which has length
$n-2$, including $P^n_{12}$ itself. Then we show a final sub-face
which has the same dimension as $P^{n-1}_{12}$, thus inductively of
dimension ${n-1 \choose 2}-(n-1)-1.$ Thus the entire flag is of
length ${n-1 \choose 2}-(n-1) + n-2 = {n \choose 2}-n.$

The base case of our induction is $n=4.$ See Figure~\ref{f:2dbmes}
where each edge of the triangle is a facet of this type.
Specifically, the edge $P^4_{12}$ is at the bottom of the triangle.

After $P^n_{12}$ the largest face in our flag is the one whose
vertices are described as vertices of $P^n_{12}$ whose caterpillar
tree is one of two types, as seen in Figure~\ref{f:cat_flag}. The
tree has a third fixed leaf, say the leaf labeled 3, either in the
same cherry as the leaf 1; or as the leaf nearest that cherry but
not in it. We call this face $P^n_{12,3}$, and note that it contains
$2((n-3)!)$ vertices. To see that it is indeed a face, we show that
its vertices obey the equation
$$x_{13}+2^{n-4}\left(\sum_{i=4}^n(x_{1i} - x_{3i})\right)=2^{n-3}.$$
...and that all other vertices in $P^n_{12}$ obey the inequality:
$$x_{13}+2^{n-4}\left(\sum_{i=4}^n(x_{1i} - x_{3i})\right)>2^{n-3}.$$

First, the vertices of $P^n_{12}$ whose caterpillar tree has the
leaf labeled 3 in the same cherry as the leaf 1: for these the
difference $x_{1i}-x_{3i} =0$ for each $i,$ while $x_{13} =
2^{n-3}$.  For the vertices that have leaf 3 as the leaf nearest the
cherry containing leaf 1, but not in it: $x_{13} = 2^{n-4}$ and the
sum of differences telescopes and simplifies to equal $2^{n-n} =1.$
The equality holds since $2(2^{n-4}) = 2^{n-3}.$

Any other leaf of $P^n_{12}$ not in $P^n_{12,3}$ has leaf 3 even
further from the cherry containing leaf 1. Now the sum of
differences will telescope and simplify to become $1+2+\dots+2^j$
where $j$ is the number of leaves further (than 1) from the cherry
that leaf 3 is found. Since the latter sum is larger than 2, the
left side of our inequality is greater than $2^{n-3}.$

Next we describe a sequence of $n-4$ nested faces (of steadily
smaller dimension) labeled $P^n_{12,34}, P^n_{12,345}, \dots ,
P^n_{12,345\dots k}$ for $k=4\dots n-1.$  The vertices of
$P^n_{12,34}$ (the first in this series, with largest dimension) are
vertices of $P^n_{12,3}$ which either have leaf 3 in the cherry with
leaf 1, or have leaf 4 in the cherry with leaf 1. After that, for
$k>4$ the vertices of $P^n_{12,34\dots k}$ are vertices of
$P^n_{12,34\dots (k-1)}$ with either leaf 3 in the cherry with leaf
1 or leaf 4 in the cherry with leaf 1 and leaves $5\dots k$ in that
order immediately on the other side of leaf 3. See
Figure~\ref{f:cat_flag}.

First we show that the vertices of $P^n_{12,345\dots k}$ obey the
equality:
$$2^{n-k}x_{13}+2^{n-4}\left(\sum_{i=4}^k(x_{1i} - x_{3i})\right)=2^{n-3}2^{n-k}.$$
Consider the vertices of $P^n_{12,345\dots k}$ whose caterpillar
tree has the leaf labeled 3 in the same cherry as the leaf 1: for
these the difference $x_{1i}-x_{3i} =0$ for each $i,$ while $x_{13}
= 2^{n-3}$.  For the vertices that have leaf 3 as the leaf nearest
the cherry containing leaf 1, but not in it: $x_{13} = 2^{n-4}$ and
the sum of differences telescopes and simplifies to equal $2^{n-k}.$
The equality holds since $2(2^{n-4}) = 2^{n-3}.$

To check for the needed inequalities we begin with $k=4.$ We show
that the vertices of $P^n_{12,3}$ which are not in $P^n_{12,34}$
obey the inequality:
$$2^{n-k}x_{13}+2^{n-4}\left(\sum_{i=4}^k(x_{1i} - x_{3i})\right)<2^{n-3}2^{n-k}.$$
For $k=4$, and since these trees have leaf 3 as the leaf nearest the
cherry containing leaf 1, but not in it, and leaf 4 also not in that
cherry, this inequality becomes:
$$2^{n-4}2^{n-4}+2^{n-4}\left(x_{14} - x_{34}\right)<
2^{n-3}2^{n-4}.$$ This inequality holds since $x_{14} - x_{34}<0,$
and so $2^{n-4} + (x_{14} - x_{34}) < 2^{n-4} < 2^{n-3},$ which
leads to the desired inequality.

Now, for $k>4,$ we need to show that the vertices of $P^n_{12,3\dots
k-1}$ which are not in $P^n_{12,3\dots k}$ obey the inequality:
$$2^{n-k}x_{13}+2^{n-4}\left(\sum_{i=4}^k(x_{1i} - x_{3i})\right)>2^{n-3}2^{n-k}.$$
Since these trees have leaf 3 as the leaf nearest the cherry
containing leaves 1 and 4, this inequality becomes:
$$2^{n-k}2^{n-4}+2^{n-4}\left(\sum_{i=4}^k(x_{1i} - x_{3i})\right)>2^{n-3}2^{n-k}.$$
Since leaf $k$ is at a position farther from leaf 1 than if the tree
was in $P^n_{12,3\dots k},$ then the sum of differences telescopes
and simplifies to $2^{n-k+1}-2^x$ where $x < n-k.$ Thus $x-(n-k) <
0$ and it is clear that $1+2-2^{x-(n-k)} >2.$ Therefore:
$$2^{n-k}2^{n-4}+2^{n-4}2^{n-k}(2-2^{x-(n-k)})>2^{n-3}2^{n-k},$$
which is the simplified inequality we needed to show.

Finally we reach the subface called $P^n_{123},$ which consists of
the vertices of $P^n_{12,345\dots n-1}$ which have leaf 3 in the
same cherry as leaf 1. That of course means they are only of the
first type. They constitute a subface, since they obey the
additional equality $x_{13} = 2^{n-3},$ while the other vertex of
$P^n_{12,345\dots n-1}$ obeys $x_{13} = 2^{n-4}<2^{n-3}.$ We check
that this face $P^n_{123}$ projects to $P^{n-1}_{12},$ which as
discussed above will give us the correct number of remaining faces
of our flag, by induction. The linear projection is described by its
action on the ${n \choose 2}$ coordinates, yielding ${n-1 \choose 2
}$ new ones; thus it is given by an ${n-1 \choose 2}\times{n \choose
2 }$ matrix $A$, which has rows and columns labeled by the
respective coordinates in lexicographic order.

First each coordinate involving leaf 3 (so $x_{i3}$ or $x_{3i}$) is
discarded. That means the columns of $A$ corresponding to these
coordinates are made up of zeroes. Second, the coordinates that
involve leaf 1 (so $x_{i1}$ or $x_{1i}$) are multiplied by 1, but
sent to the new coordinates with the same label if $i=2,$ or to the
coordinates  $x_{(i-1)1}$ or $x_{1(i-1)}$ respectively if $i\ge 4.$
Thus the columns of $A$ corresponding to these coordinates have a
single entry of 1 in the row corresponding to the new coordinate.
Finally the coordinates $x_{ij}$ involving neither leaf 1 nor leaf 3
are multiplied by $1\slash 2$ and sent to the new coordinate
$x_{i'j'}$ where $i' = i-1$ for $i>2$ and $i'=i$ for $i = 2$; and
likewise for $j'.$ Thus the columns of $A$ corresponding to these
coordinates have a single entry of $1\slash 2$ in the row
corresponding to the new coordinate. For instance here is the matrix
$A$ for $n=5$; it takes the two sample vectors listed in
Figure~\ref{f:bme5data} (the two vertices of the lowest edge in
Figure~\ref{f:birkhoff-iso}) to the two vectors labeling the lower
edge of the triangle in Figure~\ref{f:2dbmes}.

$$A = \left[\begin{array}{cccccccccc}
1&0&0&0&0&0&0&0&0&0\\
0&0&1&0&0&0&0&0&0&0\\
0&0&0&1&0&0&0&0&0&0\\
0&0&0&0&0&1/2&0&0&0&0\\
0&0&0&0&0&0&1/2&0&0&0\\
0&0&0&0&0&0&0&0&0&1/2\\
\end{array}\right]
$$

We claim that the image of $P^{n}_{123}$ under $A$ is the polytope
$P^{n-1}_{1,2}.$ This follows from the fact that the projection $A$ induces a
1-1 and onto mapping between the vertices of the two polytopes. Since there are
$(n-3)!$ vertices of each polytope, the surjective property of the mapping
implies that it is a bijection. We can easily describe the preimage of a vertex
in $P^{n-1}_{12}:$ take its caterpillar tree, attach a new branch as close to
leaf 1 as possible, and give it leaf 3. Then add 1 to increment each of the
other leaves except for leaf 2. The resulting new tree has the coordinates
required. The coordinates involving leaf 3 will be discarded, so their value
can be ignored. The leaves in our new tree are all now one node further away
from leaf 1, but using the new total number of leaves $n$ this difference is
canceled. The coordinates involving neither leaf 1 nor leaf 3 are the same
except for the factor of 2.

%Note that we have note shown combinatorial equivalence, since vertex-facet incidence would need to be shown for that.
%Instead we just showed that A is a projection, so the dimension of its domain is no less that that of its range.

\begin{figure}[b]\centering
                  \includegraphics[width=\textwidth]{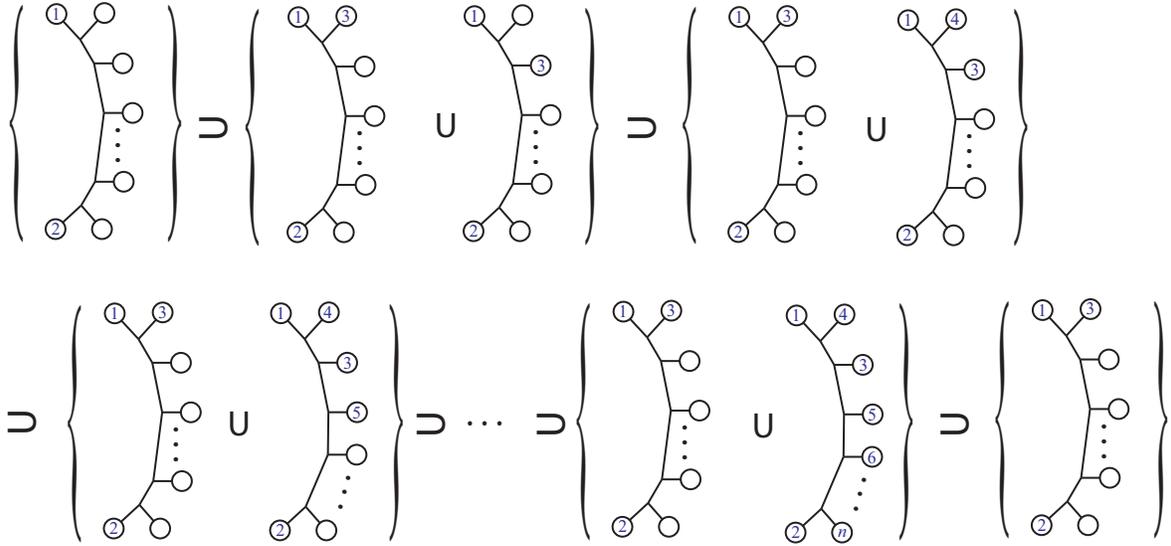}
\caption{Each set is the collection of trees (vertices of
$\mathcal{P}_n$) that have any placement of the remaining labels
from $S=[n]$ into the blank leaves shown. The containment of sets
also shows the flag of our facet $P^n_{12}$. The containment is
$P^n_{12}\supset P^n_{12,3}\supset P^n_{12,34}\supset
P^n_{12,345}\supset\dots\supset P^n_{12,345\dots n-1}\supset
P^n_{123}.$ Note that the next to last set also has leaf $n$ fixed
since there is only one place for it, and the last set (bottom
right) is the set of vertices of $\mathcal{P}_n$ which label a face
that projects to the facet $P^{n-1}_{12}.$}\label{f:cat_flag}
\end{figure}
\end{proof}

\begin{figure}[b]\centering
                  \includegraphics[width=\textwidth]{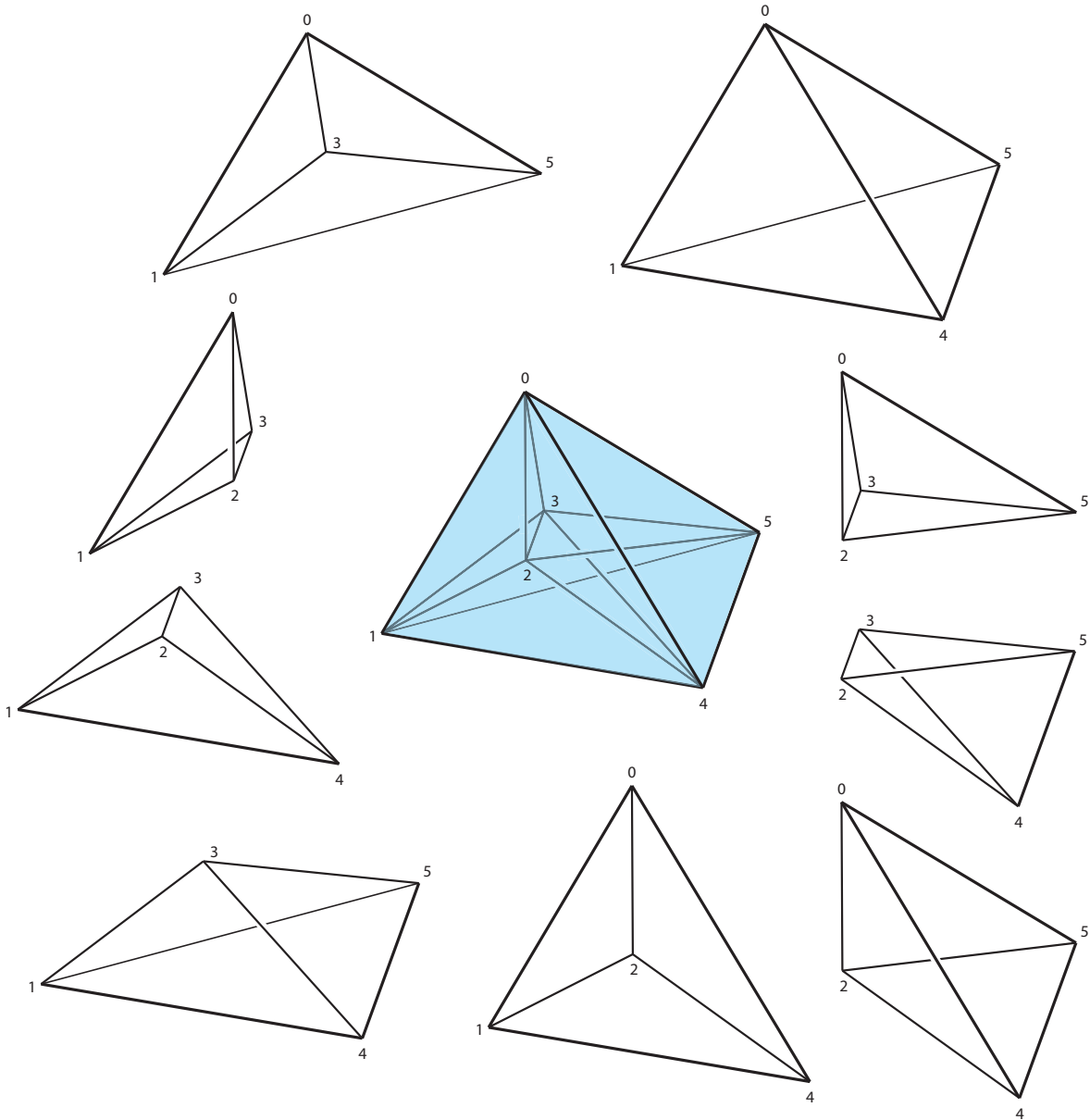}
\caption{Here we show the facets of the 4-dimensional Birkhoff
polytope. Each of the nine tetrahedra has its vertices labeled to
show where it fits in the central Schlegel diagram.}\label{f:expand}
\end{figure}

\section{Acknowledgements}
We would like to thank the referees, whose suggestions helped to improve the
readability of this paper. The first author would like to thank the organizers
and participants in the Working group for geometric approaches to phylogenetic
tree reconstructions, at the NSF/CBMS Conference on Mathematical Phylogeny held
at Winthrop University in June-July 2014. Especially helpful were conversations
with Ruriko Yoshida, Terrell Hodge and Matt Macauley. The first author would
also like to thank the American Mathematical Society and the Mathematical
Sciences Program of the National Security Agency for supporting this research
through grant H98230-14-0121.\footnote{This manuscript is submitted for
publication with the understanding that the United States Government is
authorized to reproduce and distribute reprints.} The first author's specific
position on the NSA is published in \cite{freedom}. Suffice it to say here that
he appreciates NSA funding for open research and education, but encourages
reformers of the NSA who are working to ensure that protections of civil
liberties keep pace with intelligence capabilities.

\bibliography{bme}{}
\bibliographystyle{alpha}

\end{document}